\documentclass{elsarticle}
\usepackage{amssymb}
\usepackage{bm} 
\usepackage{graphicx}
\usepackage{graphics}
\usepackage{caption2}
\usepackage{psfrag}
\usepackage{amsmath}
\usepackage{amscd}
\usepackage{amsfonts}
\usepackage{float}
\usepackage{latexsym}
\usepackage{lscape}
\usepackage{mathrsfs}

\newcommand{\ds}{\displaystyle}
\newcommand{\f}{\frac}

\usepackage{ntheorem}

\newtheorem{theorem}{Theorem}[section]

\newtheorem{lemma}[theorem]{Lemma}
\newtheorem{remark}[theorem]{Remark}
\newtheorem*{proof}[theorem]{Proof}

\journal{ }
\begin{document}

\begin{frontmatter}

\title{Numerical approximation for a nonlinear variable-order fractional differential equation via an integral equation method}

\author[1]{Xiangcheng Zheng}
\ead{zhengxch@math.pku.edu.cn}

\address[1]{School of Mathematical Sciences, Peking University, Beijing 100871, China}

\begin{abstract}
We study a numerical approximation for a nonlinear variable-order fractional differential equation via an integral equation method. Due to the lack of the monotonicity of the discretization coefficients of the variable-order fractional derivative in standard approximation schemes, existing numerical analysis techniques do not apply directly. By an approximate inversion technique, the proposed model is transformed as a second kind Volterra integral equation, based on which a collocation method under uniform or graded mesh is developed and analyzed. In particular, the error estimates improve the existing results by proving a consistent and sharper mesh grading parameter and characterizing the convergence rates in terms of the initial value of the variable order, which demonstrates its critical role in determining the smoothness of the solutions and thus the numerical accuracy.
\end{abstract}

\begin{keyword}
Nonlinear fractional differential equation, variable-order, collocation method, graded mesh, error estimate, integral equation
\end{keyword}
\end{frontmatter}

\section{Introduction}
Variable-order fractional differential equations, in which the order of the fractional operator may be a function of time, attract growing attentions in the past few decades \cite{Gar,LorHar,Sam13,WanZheJMAA,ZenZhaKar,Zha,ZhaoSun,ZhuLiu}. Nevertheless, the corresponding mathematical and numerical analysis are far from well developed, and even the results for the following simple nonlinear variable-order fractional Cauchy problem
\begin{equation}\label{ModelR}\begin{array}{c}
{}_0D_t^{\alpha(t)} u = f(u,t),~~t \in (0,T];~~u(0)=u_0
\end{array}\end{equation}
are rarely available in the literature.
Here the variable-order fractional integral operator ${}_0I_t^{1-\alpha(t)}$ and the Caputo variable-order fractional differential operator ${}_0D_t^{\alpha(t)}$ are defined by \cite{LorHar,Samko,SunFCAA}
\begin{equation}\label{IntOpr}
\ds {}_0I_t^{1-\alpha(t)}g(t):=\int_0^t\frac{1}{\Gamma(1-\alpha(s))}\frac{g(s)}{(t-s)^{\alpha(s)}}ds, ~
{}_0D_t^{\alpha(t)}g(t)={}_0I_t^{1-\alpha(t)}g'(t).
\end{equation}
The main difficulty of analyzing this problem lies in finding its equivalent second kind Volterra integral equations, which are only available for the constant-order case $\alpha(t)\equiv \alpha$ \cite{DieFor,LiYiJCP,Luc,SakYam} and for some variable-order fractional models that are naturally second kind Volterra integral equations \cite{WanZheJMAA,ZheSINUM2}. In a recent work an approximate inversion technique was proposed \cite{ZheArxiv}, which provides a potential means to convert variable-order fractional problems into equivalent second kind integral equations that significantly facilitates the analysis. Motivated by this work, we aim to present a mathematical and numerical study for model (\ref{ModelR}) based on its equivalent integral equation. 

In this work we follow the approximate inversion technique and the analysis in \cite{ZheArxiv} to convert model (\ref{ModelR}) to an equivalent second kind Volterra integral equation, based on which we prove the well-posedness and smoothing properties of the Caputo variable-order fractional Cauchy problem (\ref{ModelR}). The derived results serve as a compensation for \cite{ZheArxiv}, which focuses on the Riemann-Liouville analogue of (\ref{ModelR}).

After analyzing model (\ref{ModelR}), we intend to develop a collocation method for the equivalent integral equation. The main advantages of this method as well as our contributions are summarized as follows:
\begin{itemize}
\item[$\bullet$] It is known that the commonly-used approximation methods for the Caputo fractional derivative like the L1 scheme \cite{LinXu,SunWu} may generate non-monotonic discretization coefficients due to the impact of variable fractional order \cite{ZheSINUM2}. Therefore, traditional numerical analysis techniques do not apply. The proposed integral-equation-based collocation method could circumvent this monotonicity issue and thus significantly facilitates the error estimates. 

\item[$\bullet$] We find from the equivalent integral equation that if the variable order is linear and the partition is uniform, which is a commonly encountered case in real applications \cite{SunExp}, then the discretization coefficients of the collocation method has the translation invariant property that could be employed in numerical implementations to reduce the computational costs and memory requirements. In general, traditional approximation methods do not enjoy this benefit even in this special case, which demonstrates the advantage of the proposed integral equation method.

\item[$\bullet$] In \cite[Theorem 8.2]{ZheSINUM}, there are similar estimates for the truncation error of the proposed method in this work under uniform or graded mesh. However, (\textbf{i}) the mesh grading parameter $r$ for the case of singular solutions in \cite[Theorem 8.2]{ZheSINUM} does not converge to that for the case of smooth solutions when the singularity gradually vanishes, which implies that the choice of $r$ for the singular case is not sharp and consistent. Furthermore, (\textbf{ii}) the convergence orders depend on the bound of the variable order in \cite{ZheSINUM}, which does not reflect the critical role of $\alpha(0)$ in determining the smoothness of the solutions (cf. Theorems \ref{thm:u}--\ref{thm:u''}) and thus the convergence rates. In this work we improve the existing results by providing a sharper and consistent mesh grading $r$ for the case of singular solutions and characterizing the convergence orders in terms of the initial value $\alpha(0)$ of $\alpha(t)$. These improvements will be justified by numerical experiments.
\end{itemize}

The rest of the paper is organized as follows: In \S 2 we analyze the wellposedness and smoothing properties of the proposed model. in \S 3 we develop a collocation method for the equivalent integral equation and in \S 4 we prove its optimal-order convergence estimates under uniform or graded mesh. Some numerical experiments are carried out in \S 5 to substantiate the mathematical and numerical analysis. 
\section{Model and analysis} 
In this work we consider the Cauchy problem of a nonlinear variable-order fractional differential equation (\ref{ModelR}).
Let $C^m[0,T]$ and $C^m(0,T]$ for $0\leq m\in\mathbb N$ be the spaces of $m$-th continuously differentiable functions on $[0,T]$ and $[\varepsilon,T]$ for any $0<\varepsilon\ll 1$, respectively, equipped with standard norms \cite{AdaFou}. We then make the following assumptions on model (\ref{ModelR}):

\paragraph{\textbf{Assumption A}} $ \alpha \in C^2[0,T]$ and $ 0<\alpha_*\leq \alpha(t)<1$ on $t\in (0,T]$ for some $0<\alpha_*<1$. $f$ is second order differentiable on $\mathbb R\times [0,T]$ with 
$$\sum_{i=1}^2|f_i|+\sum_{i,j=1}^2|f_{i,j}|\leq L$$
for some $L>0$ where $f_1$ and $f_2$ refer to the derivatives of $f$ with respect to the first and the second arguments, respectively, and $f_{i,j}:=(f_i)_j=(f_j)_i$. 

In this paper, We use $Q$ to denote a generic positive constant that may assume different values at different cases.

\subsection{Auxiliary inequalities}
We present two inequalities to be used subsequently.
\begin{lemma}(Discrete Gronwall inequality \cite{Bru})\label{thm:DGron}
Suppose the non-negative sequence $\{z_n\}_{n=1}^N$ satisfies
\begin{equation}\label{Gronwall:e1}
z_n\leq \frac{Q}{N^\beta}\sum_{i=1}^{n-1}\frac{z_i}{(n-i)^{1-\beta}}+y, \quad 1\leq n\leq N,~0 < \beta < 1,~~y\geq 0.
\end{equation}
Then $z_n $ is bounded  by 
$$z_n \leq y\,(1+E_{\beta,1}(Q\Gamma(\beta))),~~1\leq n\leq N.$$
Here $E_{p,q}(z)$ represents the Mittag-Leffler function defined by \cite{KilSri,Pod}
$$\ds E_{p,q}(z) := \sum_{k=0}^{\infty}\frac{z^k}{\Gamma(pk+q)}, \qquad z \in \mathbb{R}, ~~p \in \mathbb{R}^+, ~~q \in \mathbb{R}.$$
\end{lemma}

We finally prove a useful result that could be used in the future. 
\begin{lemma}\label{lemcont}
Let $g\in C(0,T]$ satisfy $|g|\leq Qt^{-\beta}$ for some $0<\beta<1$ and the Assumption A holds. Then ${}_0I^{1-\alpha(t)}_tg\in C(0,T]$.
\end{lemma}
\begin{proof}
It suffices to prove that ${}_0I^{1-\alpha(t)}_tg\in C[\varepsilon,T]$ for any $0<\varepsilon\ll 1$. Let $\varepsilon\leq t_1<t_2\leq T$ be such that $t_2-t_1<1$. Then a direct calculation yields
\begin{equation}
\begin{array}{l}
\ds {}_0I^{1-\alpha(t_2)}_{t_2}g(t_2)-{}_0I^{1-\alpha(t_1)}_{t_1}g(t_1)=\int_{t_1}^{t_2}\frac{1}{\Gamma(1-\alpha(s))}\frac{g(s)ds}{(t_2-s)^{\alpha(s)}}\\[0.125in]
\ds\qquad+\int_0^{\varepsilon/2}\frac{g(s)}{\Gamma(1-\alpha(s))}\Big(\frac{1}{(t_2-s)^{\alpha(s)}}-\frac{1}{(t_1-s)^{\alpha(s)}}\Big)ds\\[0.125in]
\ds\qquad+\int^{t_1}_{\varepsilon/2}\frac{g(s)}{\Gamma(1-\alpha(s))}\Big(\frac{1}{(t_2-s)^{\alpha(s)}}-\frac{1}{(t_1-s)^{\alpha(s)}}\Big)ds:=\sum_{i=1}^3\bar J_i.
\end{array}
\end{equation}  
We intend to prove that this difference tends to $0$ as $t_2-t_1\rightarrow 0^+$.
We bound $\bar J_1$ by
$$|\bar J_1|\leq Q\varepsilon^{-\beta}\int_{t_1}^{t_2}\frac{1}{(t_2-s)^{\bar \alpha}}ds\leq Q\varepsilon^{-\beta}(t_2-t_1)^{1-\bar\alpha},~~\bar\alpha:=\max_{\varepsilon\leq t\leq T}\alpha(t)<1.$$
We apply the mean value theorem to bound $\bar J_2$ by
$$|\bar J_2|\leq Q\int_0^{\varepsilon/2}\frac{s^{-\beta}(t_2-t_1)}{(t_1-s)^{1+\alpha(s)}}ds\leq Q\varepsilon^{-\beta-\bar\alpha}(t_2-t_1).$$
We bound $\bar J_3$ by using the fact $(t_1-s)^{-\alpha(s)}>(t_2-s)^{-\alpha(s)}$ for any $\varepsilon/2\leq s\leq t_1$
$$\begin{array}{l}
\ds |\bar J_3|\leq Q\varepsilon^{-\beta}\int_{\varepsilon/2}^{t_1}(t_1-s)^{-\alpha(s)}-(t_2-s)^{-\alpha(s)}ds\\[0.15in]
\ds\qquad\qquad=Q\varepsilon^{-\beta}\int_{\varepsilon/2}^{t_1}(t_1-s)^{-\alpha(s)}\Big(1-\Big(\frac{t_1-s}{t_2-s}\Big)^{\alpha(s)}\Big)ds\\[0.15in]
\ds\qquad\qquad\leq Q\varepsilon^{-\beta}\int_{\varepsilon/2}^{t_1}(t_1-s)^{-\bar\alpha}\Big(1-\Big(\frac{t_1-s}{t_2-s}\Big)^{\bar\alpha}\Big)ds\\[0.15in]
\ds\qquad\qquad=Q\varepsilon^{-\beta}\int_{\varepsilon/2}^{t_1}(t_1-s)^{-\bar\alpha}-(t_2-s)^{-\bar\alpha}ds\\[0.15in]
\ds\qquad\qquad=\frac{Q\varepsilon^{-\beta}}{1-\bar\alpha}[(t_1-\varepsilon/2)^{1-\bar\alpha}-((t_2-\varepsilon/2)^{1-\bar\alpha}-(t_2-t_1)^{1-\bar\alpha})]\\[0.15in]
\ds\qquad\qquad\leq Q\varepsilon^{-\beta}(t_2-t_1)^{1-\bar\alpha}.
\end{array}$$
We incorporate the preceding estimates to finish the proof.
\end{proof}

\subsection{Analysis of model (\ref{ModelR})}
In a recent work the well-posedness and smoothing properties of the following variable-order Abel integral equation
$ {}_0I_t^{1-\alpha(t)}u(t)=f(t)$
and the corresponding Riemann-Liouville variable-order fractional differential equation
were proved via the approximate inversion technique \cite{ZheArxiv}, which reverts the variable-order fractional integral operator to the identify operator (or its multiple) added by a weak-singular integral operator. Based on that idea, we could also transform the nonlinear Caputo variable-order fractional Cauchy problem (\ref{ModelR}) to an equivalent integral equation and then perform the mathematical analysis similarly as \cite{ZheArxiv}. 

Following \cite[Theorem 2.1]{ZheArxiv}, we intend to show that the following variable-order fractional integral operator
$${}_0\tilde{I}_t^{\alpha(t)}g(t):=\int_0^t\frac{1}{\Gamma(\alpha(t))}\frac{g(s)}{(t-s)^{1-\alpha(t)}}ds $$
serves as the approximate inversion of ${}_0D_t^{\alpha(t)}$. To demonstrate this, for any $0<s<t$, we replace $t$ by $y$ in (\ref{ModelR}), multiply $1/\big(\Gamma(\alpha(t))(t-y)^{1-\alpha(t)}\big)$ on both sides of (\ref{ModelR}) and integrate the resulting equation from $0$ to $t$ to obtain
\begin{equation}\label{eq:e2}
\begin{array}{l}
\ds \frac{1}{\Gamma(\alpha(t))}\int_0^t\frac{1}{(t-y)^{1-\alpha(t)}}\int_0^y\frac{1}{\Gamma(1-\alpha(s))}\frac{u'(s)ds}{(y-s)^{\alpha(s)}}dy\\[0.125in]
\ds\quad\quad=\frac{1}{\Gamma(\alpha(t))}\int_0^t\frac{f(u(y),y)}{(t-y)^{1-\alpha(t)}}dy.
\end{array}
\end{equation}
We exchange the order of the double integrals on the left-hand side of (\ref{eq:e2}) and evaluate the interior integral analytically as follows
$$\int_s^t\frac{1}{(t-y)^{1-\alpha(t)}(y-s)^{\alpha(s)}}dy=\frac{\Gamma(\alpha(t))\Gamma(1-\alpha(s))}{\Gamma(1+\alpha(t)-\alpha(s))}(t-s)^{\alpha(t)-\alpha(s)}$$
to get
\begin{equation}\label{eq:e3}
\begin{array}{l}
\ds \int_0^tK(t,s)u'(s)ds=\frac{1}{\Gamma(\alpha(t))}\int_0^t\frac{f(u(y),y)}{(t-y)^{1-\alpha(t)}}dy,\\[0.15in]
\ds\qquad\qquad K(t,s):=\frac{(t-s)^{\alpha(t)-\alpha(s)}}{\Gamma(1+\alpha(t)-\alpha(s))}.
\end{array}
\end{equation}

Note that by Assumption A, the following properties of $(t-s)^{\alpha(t)-\alpha(s)}$ hold
\begin{equation}\label{eq:e2.5}
\begin{array}{l}
\ds (t-s)^{\alpha(t)-\alpha(s)}=e^{(\alpha(t)-\alpha(s))\ln(t-s)}=e^{\alpha'(\xi)(t-s)\ln(t-s)}\\[0.05in]\quad
\ds \in \big[e^{-Q_0\|\alpha\|_{C^1[0,T]}},e^{Q_0\|\alpha\|_{C^1[0,T]}}\big],~~Q_0:=\sup_{0\leq s<t\leq T}|(t-s)\ln(t-s)|,\\[0.15in]
\ds \lim_{s\rightarrow t^-}(t-s)^{\alpha(t)-\alpha(s)}=\lim_{s\rightarrow t^-}e^{(\alpha(t)-\alpha(s))\ln(t-s)}=1.
\end{array}
\end{equation}

Then an integration by parts for the left-hand side of (\ref{eq:e3}) yields a second kind Volterra integral equation
\begin{equation}\label{IntEqn}
\ds u(t)-\int_0^tK_s(t,s)u(s)ds-\frac{u_0t^{\alpha(t)-\alpha(0)}}{\Gamma(1+\alpha(t)-\alpha(0))}=\frac{1}{\Gamma(\alpha(t))}\int_0^t\frac{f(u(s),s)ds}{(t-s)^{1-\alpha(t)}}.
\end{equation}
The kernel $K_s(t,s)$ could be bounded by
\begin{equation}\label{WW:e1}
 \begin{array}{l}
 \ds\big| K_s(t,s)\big|\\[0.05in]
 \ds\quad=K(t,s)\Big|\frac{\alpha'(s)\Gamma'(1+\alpha(t)-\alpha(s))}{\Gamma(1+\alpha(t)-\alpha(s))}-\alpha'(s)\ln(t-s)-\frac{\alpha(t)-\alpha(s)}{t-s}\Big|\\[0.1in]
 \ds\quad\leq Q\big(1+|\ln(t-s)|\big)\leq \frac{Q}{(t-s)^\varepsilon},~~0<\varepsilon\ll 1.
 \end{array}
 \end{equation}
Based on this expression, we prove the well-posedness and smoothing properties of the variable-order fractional differential equation (\ref{ModelR}) in the following theorems. The proofs could be performed by similar techniques as those of \cite[Theorems 3.1 and 4.1]{ZheArxiv} and are thus omitted. 

\begin{theorem}\label{thm:u}
Suppose the Assumption A holds. Then for $\alpha(0)<1$, model (\ref{ModelR}) has a unique solution $u \in C[0,T]\cap C^1(0,T]$ with 
\begin{equation}\label{Wellpose1a}
\begin{array}{c}
\ds \| u \|_{C[0,T]}+\max_{t\in[0,T]}t^{1-\alpha(0)}|u'(t)| \le QM,\\[0.05in]
\ds M:= \big(|u_0|+\|f(0,\cdot)\|_{C[0,T]}\big). 
\end{array}
\end{equation}
Here $Q$ may depend on $L$, $\alpha_*$ and $\|\alpha\|_{C^2[0,T]}$.

For $\alpha(0)=1$, model (\ref{ModelR}) has a unique solution $u \in C^1[0,T]$ with
\begin{equation}\label{Wellpose1b}
\| u\|_{C^1[0,T]} \le Q M.
\end{equation}

\end{theorem}

\begin{theorem}\label{thm:u''}
	Suppose the Assumption A holds. Then if $\alpha(0)<1$, $u \in C[0,T]\cap C^2(0,T]$ with 
\begin{equation}\label{Wellpose2a}
\max_{t\in[0,T]}t^{2-\alpha(0)}|u''(t)| \le QM .
\end{equation}
Here $Q$ may depend on $L$, $\alpha_*$ and $\|\alpha\|_{C^2[0,T]}$ and $M$ is defined in (\ref{Wellpose1a}).

If $\alpha(0)=1$ and $\alpha'(0)=0$, $u \in C^2[0,T]$ with
\begin{equation}\label{Wellpose2b}
\| u\|_{C^2[0,T]} \le Q M.
\end{equation}
\end{theorem}
\begin{remark}\label{remequ}
It is worth mentioning that, after showing the well-posedness and smoothing properties (\ref{Wellpose1a})--(\ref{Wellpose1b}) of the VIE (\ref{IntEqn}) by methods in \cite{ZheArxiv}, we rewrite (\ref{IntEqn}) back to its original form as
\begin{equation}\label{MH}
 {}_0\tilde{I}_t^{\alpha(t)}\big({}_0I_t^{1-\alpha(t)}u'(t)-f(u,t) \big)=0. 
 \end{equation}
By Lemma \ref{lemcont}, the term ${}_0I_t^{1-\alpha(t)}u'(t)$ and thus ${}_0I_t^{1-\alpha(t)}u'(t)-f(u,t)$ is continuous on $(0,T]$. Then we could prove by contradiction that (\ref{MH}) implies ${}_0I_t^{1-\alpha(t)}u'(t)-f(u,t)=0$, and thus prove the equivalence between the model (\ref{ModelR}) and the VIE (\ref{IntEqn}) in the solution space used in the above theorems. This also implies that it suffices to develop numerical methods for the VIE (\ref{IntEqn}) instead of the original problem (\ref{ModelR}).
\end{remark}
\section{A collocation method}
Based on the discussions in Remark \ref{remequ}, we present a collocation method for the second kind VIE (\ref{IntEqn}).
Let $0 =: t_0<t_1<\cdots< t_N :=T$ be a graded partition of $[0,T]$ with 
$$t_i = T\bigg(\frac{i}{N}\bigg)^r,~~0\leq i\leq N,~~r\geq 1,$$
 which reduces to a uniform partition for $r=1$. Applying mean-value theorem bounds $\tau_i := t_i - t_{i-1}$ by  
\begin{equation}\label{bnd:h}
\max_{1 \le i \le N} \tau_i = \max_{1 \le i \le N} T\f{i^r - (i-1)^r}{N^r}  \leq \max_{1 \le i \le N} \frac{rT i^{r-1}}{N^r} \le \f{rT}{N}.
\end{equation}
Let $X$ be the space of piecewise-linear functions with respect to the partition. For any function $g(x)$ on $[0,T]$, we define 
$$\| g \|_{\hat L^\infty} := \max_{0 \le i \le N} | g(t_i) |.$$ Then the collocation method for (\ref{ModelR}) states as follows: find $U(t)\in X$ such that
\begin{equation}\label{ApprV}
\begin{array}{l}
\ds U(t_n) =\int_0^{t_n}K_s(t_n,s)U(s)ds+\frac{1}{\Gamma(\alpha(t_n))}\int_0^{t_n}\frac{f(U(s),s)ds}{(t_n-s)^{1-\alpha(t_n)}}\\[0.15in]
\ds\qquad\qquad\qquad\qquad+\frac{u_0t_n^{\alpha(t_n)-\alpha(0)}}{\Gamma(1+\alpha(t_n)-\alpha(0))}, \quad 0\leq n\leq N. 
\end{array}
\end{equation}

\subsection{A special case: linear variable order and uniform partition}
In practical applications, the linear variable fractional order is often used in the model to fit the experimental data due to its simplicity \cite{SunExp}. The uniform partition is also commonly used in numerical methods. We will show that in the case of linear variable order and uniform partition with the mesh size $\tau$, the discretization coefficients of the first right-hand side term of (\ref{ApprV}) exhibit translation invariant property that could be used in numerical implementations to reduce the computational costs and memory requirements. In general, traditional discretization methods for variable-order problem (\ref{ModelR}) do not enjoy this benefit even in this special case, which demonstrates the advantage of the proposed integral equation method.

Suppose $\alpha(t)$ is a linear function of $t$, we observe from (\ref{WW:e1}) that $K_s(t,s)$ is a function of $t-s$. We then apply the exact formula of $U(t)$ on each subinterval $[t_{i-1},t_i]$ to write the first right-hand side term of (\ref{ApprV}) in details as follows
\begin{equation*}
\begin{array}{l}
\ds \int_0^{t_n}K_s(t_n,s)U(s)ds\\
\ds\quad=\sum_{i=1}^n\int_{t_{i-1}}^{t_i} K_s(t_n,s)\bigg(\frac{s-t_{i-1}}{\tau}U(t_i)+\frac{t_i-s}{\tau}U(t_{i-1})\bigg)ds\\[0.15in]
\ds\quad= \int_{t_{n-1}}^{t_n} K_s(t_n,s)\frac{s-t_{n-1}}{\tau}ds U(t_n)\\[0.15in]
\ds\qquad+\sum_{i=1}^{n-1}\bigg(\int_{t_{i}}^{t_{i+1}} K_s(t_n,s)\frac{t_{i+1}-s}{\tau}ds+\int_{t_{i-1}}^{t_i} K_s(t_n,s)\frac{s-t_{i-1}}{\tau}ds \bigg)U(t_i)\\[0.175in]
\ds\qquad+\int_{t_0}^{t_1} K_s(t_n,s)\frac{t_1-s}{\tau}ds\, u_0\\[0.05in]
\ds\quad=:h_{n,n}U(t_n)+\sum_{i=1}^{n-1}h_{n,i}U(t_i)+\int_{t_0}^{t_1} K_s(t_n,s)\frac{t_1-s}{\tau}ds\, u_0.
\end{array}
\end{equation*}
By variable substitution $s\rightarrow s+\tau$ we could verify that $h_{n,i}=h_{n+1,i+1}$ for all reasonable $i$ and $n$, which shows the translation invariant property of $\{h_{n,i}\}$.

\section{Error estimate}
We prove error estimates for the collocation scheme (\ref{ApprV}) in the following three cases:
\begin{itemize}
\item[] (I) $\alpha(0)=1$, $\alpha'(0)=0$ and $r=1$;
\item[] (II) $\alpha(0)<1$ and  $r=1/\alpha(0)$;
\item[] (III)  $\alpha(0)<1$ and $r=1$.
\end{itemize}
According to Theorems \ref{thm:u}--\ref{thm:u''}, case (I) implies the smooth solution and a uniform partition, and cases (II) and (III) imply the non-smooth solution and graded or uniform partition.
We first estimate the truncation error $\mathcal R_n$, which is defined by
\begin{equation}\label{rn0}
\ds \mathcal R_n:=\int_0^{t_n}\frac{|u(s)-\hat U(s)|}{(t_n-s)^{\varepsilon_0}}ds,~~1-\alpha_*<\varepsilon_0<1
\end{equation}
for case (I), and
\begin{equation}\label{rn}
\ds \mathcal R_n:=\int_0^{t_n}\frac{|u(s)-\hat U(s)|}{(t_n-s)^{1-\alpha(t_n)}}ds=\sum_{i=1}^n\int_{t_{i-1}}^{t_i}\frac{|u(s)-\hat U(s)|}{(t_n-s)^{1-\alpha(t_n)}}ds,
\end{equation}
for cases (II) and (III).
Here $\hat U\in X$ refers to the piecewise linear interpolation of $u$, i.e., $\hat U\in X$ satisfies $\hat U(t_n)=u(t_n)$ for $0\leq n\leq N$. The reason of such definitions will be shown later. It is worth mentioning that in \cite[Theorem 8.2]{ZheSINUM}, there are estimates for a similar term as $\mathcal R_n$ under uniform or graded mesh. However, those results could be further improved from two aspects:
\begin{itemize}
\item[$\bullet$] The mesh grading parameter $r$ for the case of singular solutions in \cite[Theorem 8.2]{ZheSINUM} does not converge to that for the case of smooth solutions when the singularity gradually vanishes, which implies that the choice of $r$ in the singular case is not sharp and consistent. Therefore, we re-estimate $\mathcal R_n$ in details in the following theorem to find a sharper and consistent mesh grading parameter $r$.

\item[$\bullet$] For the case of uniform partition and singular solutions in \cite[Theorem 8.2]{ZheSINUM}, the order of the estimate depends on the bound of $\alpha$ but not on its value at the starting point (the initial value $\alpha(0)$ in the current case). From Theorems \ref{thm:u} and \ref{thm:u''} we find that the singularity of the solutions is determined by the behavior of $\alpha$ at $t=0$. Therefore, it may be possible to improve the estimate of $\mathcal R_n$ by characterizing it via $\alpha(0)$.
\end{itemize}

We address these issues to estimate $\mathcal R_n$ in the following theorem.
\begin{theorem}\label{thm:rn}
Suppose the Assumption A holds. Then for cases (I) and (II)
the following estimate holds
\begin{equation*}
\max_{1\leq n\leq N}|\mathcal R_n|\leq Q MN^{-2}.
\end{equation*}	
Here $M$ is defined in (\ref{Wellpose1a}) and $Q$ may depend on $\alpha_*$, $\|\alpha\|_{C^2[0,1]}$ and $L$.

Otherwise, in case (III)
 a sub-optimal estimate holds
\begin{equation*}
\max_{1\leq n\leq N}|\mathcal R_n|\leq Q MN^{-2\alpha(0)}.
\end{equation*}
\end{theorem}
\begin{remark}
We find that as $\alpha(0)$ tends to $1$, the mesh grading $r$ in (II) approaches $1$ and thus the mesh is close to the uniform partition. Therefore, the mesh grading $r$ in the case (II) is consistent with that in (I). We also notice that both the mesh grading $r$ and the orders of the estimates are determined by $\alpha(0)$, which again demonstrates the key role of the initial value of $\alpha(t)$. In conclusion, the proved results improve those in \cite[Theorem 8.2]{ZheSINUM} from two aspects as mentioned above that will be justified by numerical experiments.
\end{remark}
\begin{proof}
Let $G^{(1)}_i(s;t) := (t_i-t)/\tau_i$ for $s \in [t_{i-1},t]$ or $-(t-t_{i-1})/\tau_i$ for $s \in [t,t_i]$ and $G^{(2)}_i(s;t) := -(t_i-t)(s-t_{i-1})/\tau_i$ for $s \in [t_{i-1},t]$ or $-(t-t_{i-1})(t_i-s)/\tau_i$ for $s \in [t,t_i]$. It is known that the error of the linear interpolation could be expressed as
\begin{equation}\label{ErrorExp}
u(x)-\hat U(x) \Bigl |_{[t_{i-1},t_i]} \Bigr . = \int_{t_{i-1}}^{t_i} G_i^{(m)}(s;t) \f{d^m u(s)}{d^m s} ds, \quad 1 \le i \le N,~m = 1, 2. 
\end{equation}
For case (I), $u\in C^2[0,T]$ by Theorem \ref{thm:u''}. Then a standard interpolation estimate yields 
$$\ds |\mathcal R_n| \leq Q\|u\|_{C^2[0,T]}N^{-2}\int_0^{t_n}(t_n-s)^{-\varepsilon_0}ds \leq Q MN^{-2}.$$
	
For case (II), we use (\ref{eq:e2.5}), (\ref{Wellpose1a}),  (\ref{ErrorExp}) with $m=1$, $t_1 =T N^{-r}$ and $t_n-t_1\geq t_1$ for $n> 1$ as well as
$$ \frac{1}{(t_n-t_1)^{1-\alpha(t_n)}}\leq \frac{Q}{(t_n-t_1)^{1-\alpha(t_1)}}\leq \frac{Q}{t_1^{1-\alpha(t_1)}}=\frac{Q}{t_1^{1-\alpha(0)}}\frac{Q}{t_1^{\alpha(0)-\alpha(t_1)}}\leq \frac{Q}{t_1^{1-\alpha(0)}} $$
 to bound the integral on the first interval $[0,t_1]$ in (\ref{rn}) by
$$\begin{array}{l}
\ds \int_0^{t_1}\frac{|u(s)-\hat U (s)|}{(t_n-s)^{1-\alpha(t_n)}}ds \leq \int_0^{t_1}\frac{\int_0^{t_1} | u'(y) | dy}{(t_n-s)^{1-\alpha(t_n)}}ds \\[0.15in]
\quad \ds \le Q M \int_0^{t_1}\frac{\int_0^{t_1} y^{\alpha(0)-1} dy}{(t_n-s)^{1-\alpha(t_n)}}ds  \leq QMt_1^{\alpha(0)}\int_0^{t_1}\frac{1}{(t_n-s)^{1-\alpha(t_n)}}ds\\[0.15in]
\quad \ds \le\left\{
\begin{array}{l}
\ds Q M t_1^{\alpha(0)}t_1^{\alpha(1)}\leq Q M t_1^{2\alpha(0)},~~\qquad\qquad~~\, n=1,\\[0.075in]
\ds QMt_1^{\alpha(0)}\frac{t_1}{(t_n-t_1)^{1-\alpha(t_n)}}\leq QMt_1^{2\alpha(0)},~~n>1
\end{array}
\right.\\[0.25in]
\quad\ds
\leq
QMN^{-2r\alpha(0)}.
\end{array}$$	

We use (\ref{ErrorExp}) with $m=2$ and (\ref{Wellpose2a}) to bound the integral on $[t_{i-1},t_i]$ for $2 \le i \le n$ in (\ref{rn}) 
\begin{equation}\label{trun:e3}\begin{array}{l}
\ds  \int_{t_{i-1}}^{t_i}\frac{|u(s)-\hat U(s)|}{(t_n-s)^{1-\alpha(t_n)}}ds  \leq \tau_i 
\int_{t_{i-1}}^{t_i}\frac{\int_{t_{i-1}}^{t_i}|u''(y)|dy}{(t_n-s)^{1-\alpha(t_n)}}ds \\[0.15in]
\ds \qquad\qquad\leq Q M \tau_i \int_{t_{i-1}}^{t_i}\frac{\int_{t_{i-1}}^{t_i}y^{\alpha(0)-2} dy}{(t_n-s)^{1-\alpha(t_n)}}ds\\[0.15in]
\ds \qquad\qquad\leq QM\tau_i^2t_{i-1}^{\alpha(0)-2}\int_{t_{i-1}}^{t_i}(t_n-s)^{\alpha(t_n)-1}ds \\[0.15in]
\ds \qquad\qquad\qquad\le Q\tau_i^2t_{i-1}^{\alpha(0)-2} \big ((t_n-t_{i-1})^{\alpha(t_n)}-(t_n-t_i)^{\alpha(t_n)}\big).
\end{array}\end{equation} 
In the rest of the proof we will consider \begin{equation}\label{rran}
 r=1\text{ or } r=\frac{1}{\alpha(0)},
 \end{equation}
  which implies $r\alpha(0)\leq 1$.
We use (\ref{bnd:h}) and (\ref{rran}) to bound (\ref{trun:e3}) with $i=n$ by 	
$$\begin{array}{rl}
&\ds  \int_{t_{n-1}}^{t_n}\frac{|u(s)-\hat U(s)|}{(t_n-s)^{1-\alpha(t_n)}}ds  \\
&\quad\ds\leq QM \tau_n^{2+\alpha(t_n)}t_{n-1}^{\alpha(0)-2}\leq QM\frac{n^{(2+\alpha(t_n))(r-1)}}{N^{(2+\alpha(t_n))r}}\frac{(n-1)^{(\alpha(0)-2)r}}{N^{(\alpha(0)-2)r}}
\\[0.15in]
&\quad\ds= QM\frac{n^{r(\alpha(0)+\alpha(t_n))-(2+\alpha(t_n))}}{N^{r(\alpha(0)+\alpha(t_n))}}\leq QMt_n^{\alpha(0)+\alpha(t_n)}n^{-(2+\alpha(t_n))}\\[0.15in]
&\quad\ds\leq QMt_n^{2\alpha(0)}n^{-(2+\alpha(t_n))}\leq QM\frac{n^{2r\alpha(0)-2-\alpha(t_n)}}{N^{2r\alpha(0)}}\leq QM\frac{1}{N^{2r\alpha(0)}}.
\end{array}$$
	
We use (\ref{bnd:h}) and the facts that $t_i \ge 2^{-r} t_n$ for $\lceil n/2 \rceil \le i \le n$ and that $\tau_i$ is increasing to bound 
$$\begin{array}{l}
\ds  \int_{t_{\lceil n/2 \rceil}}^{t_{n-1}} \frac{|u(s)-\hat U(s)|}{(t_n-s)^{1-\alpha(t_n)}}ds \\[0.1in]
\ds \quad \leq QM\sum_{i=\lceil n/2\rceil+1 }^{n-1}\tau_i^2t_{i-1}^{\alpha(0)-2} \big ((t_n-t_{i-1})^{\alpha(t_n)}-(t_n-t_i)^{\alpha(t_n)}\big) \\[0.2in]
\ds \quad \leq QM t_{n}^{\alpha(0)-2}\tau_n^2(t_n-t_{\lceil n/2\rceil})^{\alpha(t_n)}\leq  QMt_{n}^{\alpha(0)+\alpha(t_n)-2}\tau_n^2\\[0.05in]
\ds \quad\leq  QMt_{n}^{2\alpha(0)-2}\tau_n^2 \leq QM \frac{n^{r(2\alpha(0)-2)}}{N^{r(2\alpha(0)-2)}}\frac{n^{2(r-1)}}{N^{2r}}\\[0.05in]
\ds \quad= QM\frac{n^{2r\alpha(0)-2}}{N^{2r\alpha(0)}}\leq \frac{QM}{N^{2r\alpha(0)}}.
\end{array}$$ 
We use (\ref{bnd:h}), (\ref{rran}), the mean value theorem and the fact that $(t_n - t_i)^{1-\alpha(t_n)} \le Q t_n^{1-\alpha(t_n)}$ for $1 \le i \le \lceil n/2 \rceil$ to bound 
\begin{equation}\label{mhmh} \begin{array}{l}
\ds  \int_{t_1}^{t_{\lceil n/2 \rceil}} \frac{|u(s)-\hat U(s)|}{(t_n-s)^{1-\alpha(t_n)}}ds \\[0.15in]
\ds\qquad\leq QM\sum^{\lceil n/2\rceil }_{i=2}\tau_i^2t_{i-1}^{\alpha(0)-2}(t_n-t_i)^{\alpha(t_n)-1}\tau_i \\[0.15in]
\qquad\ds \leq QMt_n^{\alpha(t_n)-1}\sum^{\lceil n/2\rceil }_{i=2}t_{i-1}^{\alpha(0)-2}\tau_i^3\\[0.15in]
\qquad\ds \leq QMt_n^{\alpha(0)-1}\sum^{\lceil n/2\rceil }_{i=2}\frac{(i-1)^{r(\alpha(0)-2)}}{N^{r(\alpha(0)-2)}}\frac{i^{3(r-1)}}{N^{3r}}\\[0.15in]
\ds \qquad\leq\frac{QMn^{r(\alpha(0)-1)}}{N^{2r\alpha(0)}}\sum^{\lceil n/2\rceil }_{i=2} i^{r(\alpha(0)+1)-3}.
\end{array}\end{equation}
If $r=1$, then the right-hand side of (\ref{mhmh}) is bounded by
$$\frac{QMn^{r(\alpha(0)-1)}}{N^{2r\alpha(0)}}\sum^{\lceil n/2\rceil }_{i=2} i^{\alpha(0)-2}\leq \frac{QMn^{r(\alpha(0)-1)}}{N^{2r\alpha(0)}}\leq \frac{QM}{N^{2r\alpha(0)}} .$$ 
If $r=1/\alpha(0)(>1)$, then the right-hand side of (\ref{mhmh}) is bounded by
$$\frac{QMn^{r(\alpha(0)-1)}}{N^{2r\alpha(0)}}\sum^{\lceil n/2\rceil }_{i=2} i^{r-2}\leq \frac{QMn^{r(\alpha(0)-1)+r-1}}{N^{2r\alpha(0)}}=\frac{QM}{N^{2r\alpha(0)}}. $$
We collect the preceding estimates to complete the proof of case (II). The last estimate of this theorem is also a consequence of the above derivations. 	
\end{proof}

\begin{theorem}\label{thm:v} 
Suppose the Assumption A holds.  
Then for $N$ sufficiently large, the optimal-order error estimate holds for scheme (\ref{ApprV}) for cases (I) and (II)
\begin{equation}\label{thmv:e0}
\| u- U\|_{\hat L^\infty} \leq QMN^{-2}.
\end{equation}	
Here $M$ is defined in (\ref{Wellpose1a}) and $Q=Q(\alpha_*,\|\alpha\|_{C^2[0,1]},L)$.

For case (III), a sub-optimal estimate holds
\begin{equation*}
\| u- U\|_{\hat L^\infty} \leq QMN^{-2\alpha(0)}.
\end{equation*}
\end{theorem}

\begin{proof} Let $E:=\hat U-U\in X$ such that $E(t_n)=u(t_n)-U(t_n)$. Subtracting (\ref{ApprV}) from (\ref{IntEqn}) yields an error equation 
$$\begin{array}{rl}
\ds E(t_n) \hspace{-0.1in}&\ds= \int_0^{t_n}K_s(t_n,s)(u(s)-U(s))ds+\int_0^{t_n}\frac{(f(u(s),s)-f(U(s),s))ds}{\Gamma(\alpha(t_n))(t_n-s)^{1-\alpha(t_n)}}\\[0.15in]
&\ds=\int_0^{t_n}K_s(t_n,s)(\hat U(s)- U(s))ds+\int_0^{t_n}\frac{(f(\hat U(s),s)-f(U(s),s))ds}{\Gamma(\alpha(t_n))(t_n-s)^{1-\alpha(t_n)}}\\[0.15in]
&\ds~~+\int_0^{t_n}K_s(t_n,s)(u(s)-\hat U(s))ds+\int_0^{t_n}\frac{(f(u(s),s)-f(\hat U(s),s))ds}{\Gamma(\alpha(t_n))(t_n-s)^{1-\alpha(t_n)}}.
\end{array} $$
For case (I), we invoke (\ref{WW:e1}) with $\varepsilon=\varepsilon_0$ ($\varepsilon_0$ is given in (\ref{rn0})) and 
$$\frac{1}{(t_n-s)^{1-\alpha(t_n)}}=\frac{(t_n-s)^{\varepsilon_0-(1-\alpha(t_n))}}{(t_n-s)^{\varepsilon_0}}\leq \frac{Q}{(t_n-s)^{\varepsilon_0}} $$
in the error equation to obtain
$$|E(t_n) |\leq Q
\int_{0}^{t_n}\frac{|E(s)|ds}{(t_n-s)^{\varepsilon_0}} + Q|\mathcal R_n| .$$
For cases (II) and (III), $\alpha(0)<1$ implies that $\alpha(t)$ is bounded away from $1$ by some positive constant $\alpha^*<1$. Therefore, we take $\varepsilon=1-\alpha^*$ in (\ref{WW:e1}) to obtain
$$|K_s(t,s)|\leq \frac{Q}{(t-s)^{1-\alpha^*}}=\frac{Q(t-s)^{\alpha^*-\alpha(t)}}{(t-s)^{1-\alpha(t)}}\leq \frac{Q\max\{1,T\}}{(t-s)^{1-\alpha(t)}}. $$
Thus we get from the error equation that
$$|E(t_n) |\leq Q
\int_{0}^{t_n}\frac{|E(s)|ds}{(t_n-s)^{1-\alpha(t)}} + Q|\mathcal R_n| \leq Q
\int_{0}^{t_n}\frac{|E(s)|ds}{(t_n-s)^{1-\alpha_*}} + Q|\mathcal R_n| .$$

The rest of the proof could be performed following that of \cite[Theorem 5.1]{ZheSINUM} by using the Gronwall inequality in Lemma \ref{thm:DGron}, and is thus omitted.
\end{proof}

\section{Numerical experiments}
We numerically substantiate the mathematical and numerical analysis in previous sections. To ensure the accuracy, the Legendre-Gauss-Lobatto numerical quadrature formula (see e.g., \cite[\S 3.3.2]{ShenTang}) with 80 nodes on each subinterval $[t_{i-1},t_i]$ was used to compute the temporal discretization coefficients in (\ref{ApprV}).

\subsection{Behavior of the solutions}
We numerically investigate the regularity of the solutions to the variable-order fractional Cauchy problem (\ref{ModelR}) and its dependence on $\alpha(0)$. We set $[0,T] = [0,1]$, $u_0=1$, $f=1$, and the variable order $\alpha(t)$ is given
\begin{equation}\label{alpha1}
\alpha(t) = \alpha(1) + (\alpha(0) - \alpha(1))\Big((1-t) - \frac{\sin(2\pi (1-t))}{2\pi} \Big).
\end{equation}
 We present the curves of this variable fractional order and the numerical solutions $U(t)$ to model (\ref{ModelR}) in the left and right plots of Figure \ref{figure:2}, respectively, for the following three cases:
\begin{equation}\label{alpha2}\begin{array}{c}
\mathrm{(i)} ~\alpha(0)=1.0, ~\alpha(1)=0.1; \quad \mathrm{(ii)} ~\alpha(0)=0.6, ~\alpha(1)=0.1; \\[0.1in]
\quad \mathrm{(iii)} ~\alpha(0)=0.3, ~\alpha(1)=0.1.
\end{array}\end{equation}
In the numerical simulations, we choose $N=1440$. We observe that the numerical solution of case (i) is smooth near the initial time $t=0$, while those for cases (ii) and (iii) exhibit singularities near $t=0$ and the singularity gets stronger as $\alpha(0)$ decreases. This coincides with the mathematical analysis in Theorems \ref{thm:u}--\ref{thm:u''}.

\begin{figure}[H]
	\setlength{\abovecaptionskip}{0pt}
	\centering
\includegraphics[width=2.25in,height=2in]{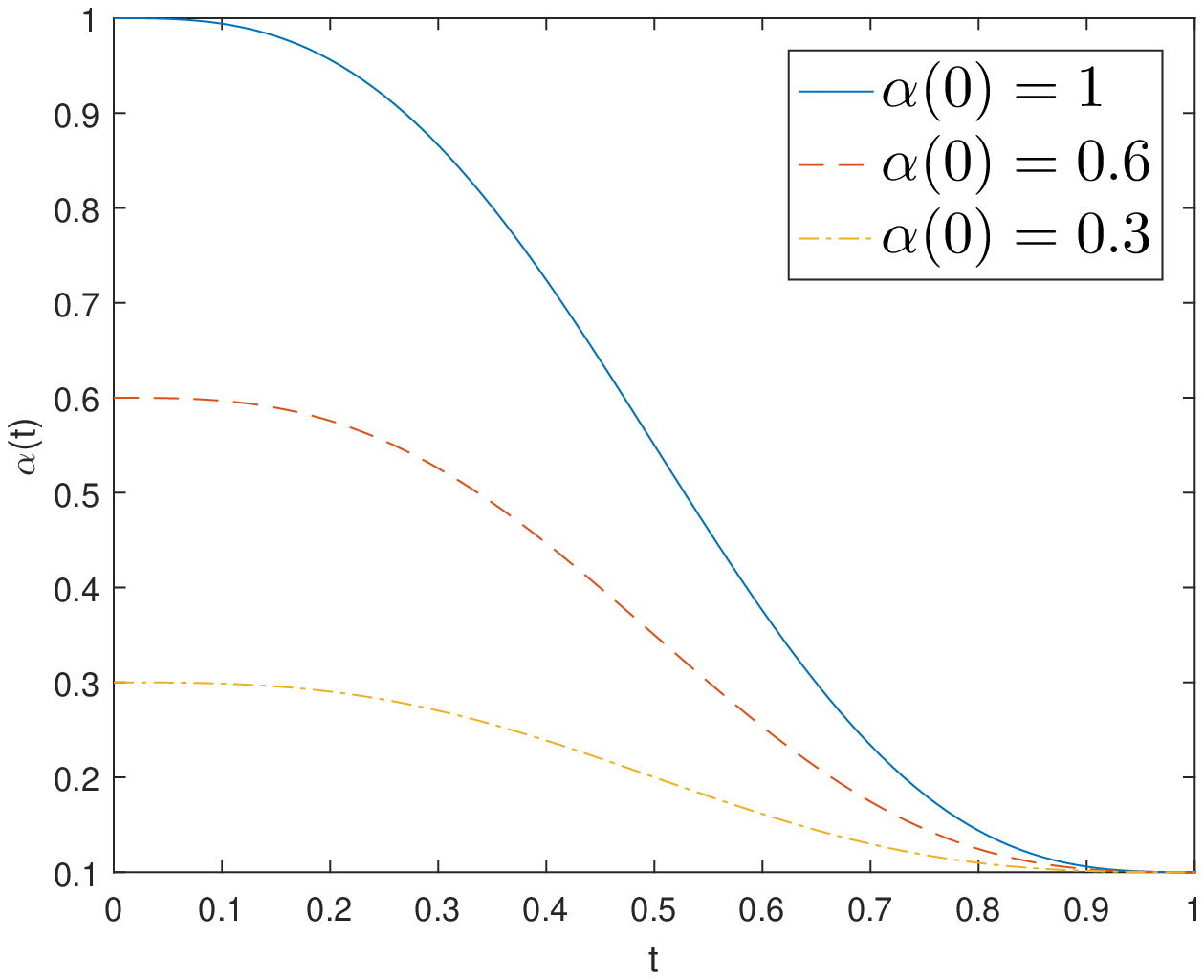}	\includegraphics[width=2.25in,height=2in]{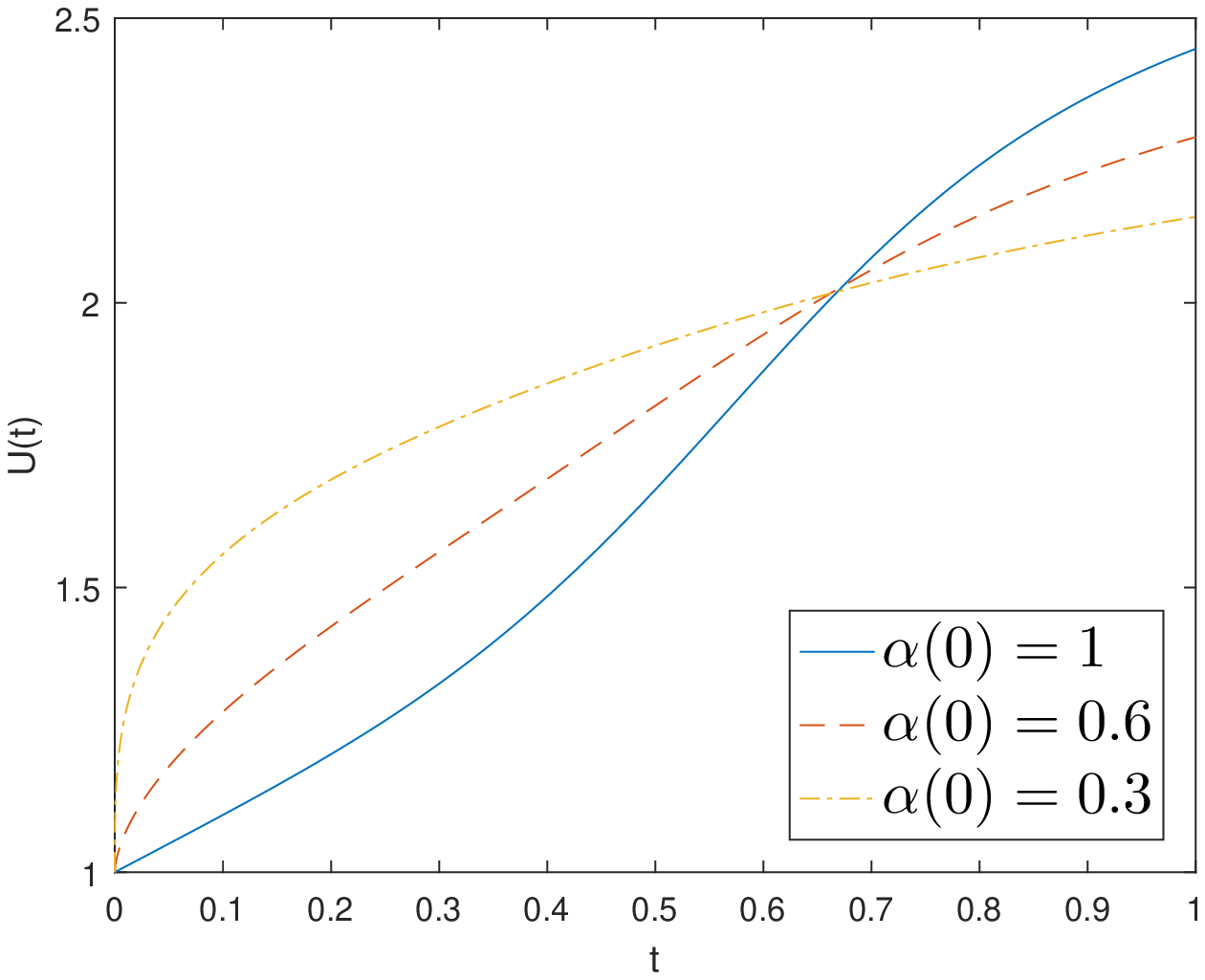}
	\caption{Plots of variable orders (left) and the numerical solutions $U(t)$ to (\ref{ModelR}) (right) for (i)--(iii).}
	\label{figure:2}
\end{figure}

\subsection{Convergence rates}
We investigate the convergence behavior of the numerical approximations to the variable-order fractional differential equation (\ref{ModelR}). Let $[0,T] = [0,1]$, $u_0=1$, $f(u)=0.5\times \sin^4(u)$ and $\alpha(t)$ be given by (\ref{alpha1}). As the exact solutions are not available, we use the numerical solutions $U_*$ discretized with $N=1440$ and either a uniform mesh or a graded mesh
of $r=1/\alpha(0)$ as the reference solutions. We measure the convergence rate $\kappa$ of the numerical approximations by
\begin{equation}\label{error}
\big \| U - U_* \big \|_{\hat L^\infty} \leq QN^{-\kappa}. 
\end{equation}
 The Newton iterative method (see e.g., \S 7.1.1 in \cite{Qua}) with the tolerance of $1\times 10^{-10}$ is used to treat the nonlinearity and we observe from Tables \ref{table0:0}--\ref{table0:1} that in the case of $\alpha(0) = 1$ a second-order convergence rate was reached under the uniform partition. However, if $\alpha(0)<1$, a uniform temporal partition leads only to a sub-optimal convergence rate $2\alpha(0)$. The second-order accuracy could be recovered by using the graded partition of $r = 1/\alpha(0)$. All these observations substantiate the theoretical analysis in Theorem \ref{thm:v}. 

\begin{table}[H]
	\setlength{\abovecaptionskip}{0pt}
	\centering
	\caption{Convergence rates under the uniform mesh with different $(\alpha(0),\alpha(1))$.}
	\vspace{0.5em}	
		\begin{tabular}{ccccccc}
			\cline{1-7}
			$1/N$&$(1,0.8)$ & $\kappa$ &$(0.6,0.4)$ &$\kappa$&$(0.4,0.2)$ &$\kappa$ \\
			\cline{1-7}		
			1/48&	1.11E-05&		&1.98E-04	&	&       1.41E-03&	\\
			1/72&	5.04E-06&	1.96&	1.20E-04&	1.22&	9.89E-04&	0.87\\
			1/96&	2.88E-06&	1.94&	8.46E-05&	1.23&	7.70E-04&	0.87\\
			1/120&	1.87E-06&	1.93&	6.42E-05&	1.24&	6.33E-04&	0.88\\
			\hline
		\end{tabular}
		\label{table0:0}
	\end{table}

\begin{table}[H]
	\setlength{\abovecaptionskip}{0pt}
	\centering
	\caption{Convergence rates under the graded mesh of $r=1/\alpha(0)$ with different $(\alpha(0),\alpha(1))$.}
	\vspace{0.5em}	
		\begin{tabular}{ccccc}
			\cline{1-5}
			$1/N$&$(0.6,0.4)$ &$\kappa$&$(0.4,0.2)$ &$\kappa$ \\
			\cline{1-5}		
		1/48&	1.05E-05&		&1.93E-05	&\\
		1/72&	4.61E-06&	2.03&	8.65E-06&	1.98\\
		1/96&	2.55E-06&	2.06&	4.88E-06&	1.99\\
		1/120&	1.61E-06&	2.05&	3.12E-06&	2.00\\
			\hline
		\end{tabular}
		\label{table0:1}
	\end{table}

\section*{Acknowledgements}

This work was partially funded by the International Postdoctoral Exchange
Fellowship Program (Talent-Introduction Program) YJ20210019 and by the
China Postdoctoral Science Foundation 2021TQ0017.



\begin{thebibliography}{}
\bibitem{AdaFou} R.A.~Adams and J.J.F.~Fournier, {\it Sobolev Spaces}, Elsevier, San Diego, 2003.


\bibitem{Bru} H. Brunner, {\it Collocation methods for Volterra integral and related functional differential equations}. Cambridge University Press, 2004.





\bibitem{DieFor} K.~Diethelm and N.~Ford. Analysis of fractional differential equations. {\it J. Math. Anal. Appl.} 265 (2002), 229--248.






\bibitem{Gar} R. Garrappa, A. Giusti, F. Mainardi, Variable-order fractional calculus: A change of perspective. {\it
Commun. Nonlinear Sci. Numer. Simul.}
 102 (2021),
105904.

\bibitem{KilSri}A.~Kilbas, H.~Srivastava and J.~Trujillo, Theory and applications of fractional differential equations, 204. Elsevier B.V., 2006.





\bibitem{LiYiJCP} C. Li, Q. Yi and A. Chen, Finite difference methods with non-uniform meshes for nonlinear fractional differential equations. {\it J Comput. Phys.} 316 (2016) 614--631. 

\bibitem{LinXu} Y. Lin, C. Xu, Finite difference/spectral approximations for the time-fractional diffusion equation. {\it J. Comput. Phys.} 225 (2007), 1533--1552.


\bibitem{LorHar} C. F. Lorenzo and T. T. Hartley, Variable order and distributed order fractional operators. {\it Nonlinear dynamics}, 29 (2002): 57--98.

\bibitem{Luc} Y. Luchko, Initial-boundary-value problems for the one-dimensional time-fractional diffusion equation. Fract. Calc. Appl. Anal. 15 (2012), 141--160.





\bibitem{Pod} I. Podlubny, {\em Fractional Differential Equations}, Academic Press, 1999.

\bibitem{Qua} A. Quarteroni, R. Sacco and F. Saleri, {\it Numerical mathematics}. Vol. 37. Springer Science \& Business Media, 2010.

\bibitem{SakYam} K.~Sakamoto and M.~Yamamoto, Initial value/boundary value problems for fractional diffusion-wave equations and applications to some inverse problems. J. Math. Anal. Appl. 382 (2011), 426--447. 


\bibitem{Samko} S. Samko and B. Ross, Integration and differentiation to a variable fractional
order. {\it Integr. Transf. Spec. Funct.} 1 (1993) 277--300.

\bibitem{Sam13} S. Samko, Fractional integration and differentiation of variable order: an
overview. {\it Nonlinear Dyn.} 71 (2013), 653--662.



\bibitem{ShenTang} J. Shen, T. Tang and L. Wang, {\it Spectral Methods: Algorithms, Analysis and Applications}. Springer Series in Compuational Mathematics, Vol. 41, Springer, 2011.



\bibitem{SunFCAA} H. Sun, A. Chang, Y. Zhang, W. Chen, A review on variable-order fractional differential equations: mathematical foundations, physical models, numerical methods and applications. {\it Fract. Calc. Appl. Anal.} 22 (2019) 27--59.

\bibitem{SunExp} H. Sun, W. Chen, H. Sheng, Y. Chen, On mean square displacement behaviors of anomalous diffusions with variable and random orders. {\it
Physics Letters A} 374 (2010), 906--910.







\bibitem{SunWu} Z. Sun, X. Wu,
A fully discrete difference scheme for a diffusion-wave system.
{\it Appl. Numer. Math.}
 56 (2006), 193--209.
 
\bibitem{WanZheJMAA} H.~Wang and X.~Zheng, Wellposedness and regularity of the variable-order time-fractional diffusion equations, {\it J. Math. Anal. Appl.}, 475 (2019): 1778--1802.




\bibitem{ZenZhaKar} F.~Zeng, Z.~Zhang and G.~Karniadakis. A generalized spectral collocation method with tunable accuracy for variable-order fractional differential equations. {\it SIAM Sci. Comp.}, 37 (2015), A2710--A2732.

\bibitem{Zha} J. Zhang, Z. Fang, H. Sun, Exponential-sum-approximation technique for variable-order time-fractional diffusion equations. {\it J. Appl. Math. Comput.} (2021) https://doi.org/10.1007/s12190-021-01528-7

\bibitem{ZhaoSun} X. Zhao, Z. Sun and G. E. Karniadakis, Second-order approximations for variable order fractional derivatives: Algorithms and applications. {\it J Comput. Phys.} 293 (2015) 184--200.


\bibitem{ZhuLiu} P.~Zhuang, F.~Liu, V.~Anh and I.~Turner. Numerical methods for the variable-order fractional advection-diffusion equation with a nonlinear source term. {\it SIAM Numer. Anal.}, 47 (2009), 1760--1781.


\bibitem{ZheSINUM} X. Zheng and H. Wang, An optimal-order numerical approximation to variable-order space-fractional diffusion equations on uniform or graded meshes. {\it SIAM J. Numer. Anal.} 58 (2020) 330--352.

\bibitem{ZheSINUM2} X. Zheng and H. Wang, An error estimate of a numerical approximation to a hidden-memory
variable-order space-time fractional diffusion equation. {\it SIAM J. Numer. Anal.} 58 (2020), 2492--2514.

\bibitem{ZheArxiv} X. Zheng, Approximate inversion for Abel integral operators of variable exponent and applications to fractional Cauchy problems. arXiv: 2110.00752.

\end{thebibliography}
\end{document}